\pdfoutput=1
\RequirePackage{ifpdf}
\ifpdf 
\documentclass[pdftex]{sigma}
\else
\documentclass{sigma}
\fi

\usepackage{mathrsfs,enumerate}

\begin{document}

\newcommand{\tr}{\operatorname{tr}}

\newcommand{\dvol}{d\operatorname{Vol}}

\newcommand{\arXivNumber}{1401.5462}

\allowdisplaybreaks

\renewcommand{\PaperNumber}{083}

\FirstPageHeading

\ShortArticleName{Generalised Chern--Simons Theory and ${\rm G}_2$-instantons over Associative Fibrations}

\ArticleName{Generalised Chern--Simons Theory\\
and $\boldsymbol{{\rm G}_2}$-Instantons over Associative Fibrations}

\Author{Henrique N.~S\'A EARP}

\AuthorNameForHeading{H.N.~S\'a Earp}

\Address{Imecc - Institute of Mathematics, Statistics and Scientific Computing, Unicamp, Brazil}
\Email{\href{mailto:henrique.saearp@ime.unicamp.br}{henrique.saearp@ime.unicamp.br}}
\URLaddress{\url{http://www.ime.unicamp.br/~hqsaearp/}}

\ArticleDates{Received January 29, 2014, in f\/inal form August 07, 2014; Published online August 11, 2014}

\Abstract{Adjusting conventional Chern--Simons theory to ${\rm G}_2$-manifolds, one describes ${\rm G}_2$-instantons on
bundles over a~certain class of $7$-dimensional f\/lat tori which f\/iber non-trivially over $T^4$, by a~pullback argument.
Moreover, if $c_2\neq0$, any (generic) deformation of the ${\rm G}_2$-structure away from such a~f\/ibred structure causes
all instantons to vanish.
A~brief investigation in the general context of (conformally compatible) associative f\/ibrations $f:Y^7\to X^4$ relates
${\rm G}_2$-instantons on pullback bundles $f^*E\to Y$ and self-dual connections on the bundle $E\to X$ over the base,
a~fact which may be of independent interest.}


\Classification{53C07; 53C38; 58J28}

\vspace{-2mm}

\section{Introduction}

This article f\/its in the context of gauge theory in higher dimensions, following the seminal works of S.~Donaldson \&
R.~Thomas, G.~Tian and others~\cite{Donaldson1998,Tian2000}.
The common thread to such generalisations is the presence of a~closed dif\/ferential form on the base manifold~$Y$,
inducing an analogous notion of anti-self-dual connections, or \emph{instantons}, on bundles over~$Y$.
In the case at hand, ${\rm G}_2$-manifolds are $7$-dimensional Riemannian manifolds with holonomy in the Lie group ${\rm
G}_2$, which implies the existence of precisely such a~structure.
This allows one to make sense of ${\rm G}_2$-\emph{instantons} as the energy-minimising gauge classes of connections,
solutions to the corresponding Yang--Mills equation.

Heuristically, ${\rm G}_2$-instantons are somewhat analogous to f\/lat connections in dimension~3.
Given a~bundle over a~compact $3$-manifold, with space of connections $\mathcal{A}$ and gauge group $\mathcal{G}$, the
\emph{Chern--Simons functional} is a~multi-valued real function on the quotient $\mathcal{B=A}/\mathcal{G}$, with
integer periods, whose critical points are precisely the f\/lat connections~\cite[\S~2.5]{Donaldson2002}.
Similar theories can be formulated in higher dimensions in the presence of a~suitable closed dif\/ferential
form~\cite{Donaldson1998,Thomas1997}; e.g.\ on a~${\rm G}_2$-manifold $(Y,\varphi)$,
the coassociative $4$-form $ \ast \varphi $ allows for the def\/inition of
a~functional of Chern--Simons type\footnote{In fact only the condition $d\ast\varphi=0$ is required, so the discussion
extends to cases in which the ${\rm G}_2$-structure~$\varphi$ is not necessarily torsion-free.}.
Its `gradient', the Chern--Simons 1-form, vanishes precisely at the ${\rm G}_2$-instantons, hence it detects the
solutions to the Yang--Mills equation.
These gauge-theoretic preliminaries are covered in Section~\ref{sec: gauge theory on G2-mfds}.

On the other hand, one may understand ${\rm G}_2$-manifolds as a~particular case of the rich theory of calibrated
geometries~\cite{Harvey1982}, for which the ${\rm G}_2$-structure~$\varphi$ is a~calibration 3-form.
Then a~3-dimensional submanifold~$P$ is said to be \emph{associative} if it is calibrated by~$\varphi$, i.e., if
$\varphi |_{P} = \dvol |_{P}$.
The deformation theory of associative submanifolds is known to be obstructed~\cite{McLean1998}, so their occurrence in
families, e.g.\ f\/ibering over a~4-manifold, is nongeneric and somewhat exotic.
Nonetheless, we may consider theoretically, at f\/irst, the existence of instantons over associative f\/ibrations $f:Y^7\to X^4$.
Given a~bundle $E\to X$, a~connection $\mathbf{A}$ on its pullback $\mathbf{E}$ is locally of the form
\begin{gather*}
\mathbf{A}\overset{\rm loc}{=} A_t(x)+\sum\limits_{i=1}^3 \sigma_i(x,t)dt^i,
\end{gather*}
where $ \{A_t \}$ is a~family of connections on~$E$ parametrised by the associative f\/ibers $P_x:=f^{-1}(x)$ and
$\sigma_i\in\Omega^0(Y,f^*\mathfrak{g}_E)$.
In Section~\ref{subsec: associative fibrations} I prove the following relation between ${\rm G}_2$-instantons and
families of self-dual connections over the base:
\begin{theorem}
\label{thm: G2-instantons over assoc fibrations}
Let $f:Y\to X$ define an associative fibration and $\mathbf{E}\to Y$ be the pullback from an indecomposable vector
bundle $ E\to X$.
\begin{enumerate}[$(i)$]\itemsep=0pt
\item If a~connection $\mathbf{A}$ on $\mathbf{E}$ is a~${\rm G}_2$-instanton, then $\{A_t\}$ is a~family of
self-dual connections on~$E$, satisfying
\begin{gather*}
\frac{\partial A_t}{\partial t^i}=d_{A_t}\sigma_i,
\qquad
i=1,2,3.
\end{gather*}
\item If, moreover, the family $A_t\equiv A_{t_0}$ is constant, then $\mathbf{A}=f^*A_{t_0}$ is a~pullback.
\end{enumerate}
\end{theorem}

 NB: We denote henceforth by $\mathcal{M}^{4}_+$ the moduli space of SD connections on the base and by~$\mathcal{M}^7_{\varphi}$ the moduli space of ${\rm G}_2$-instantons relative to ${\rm G}_2$-structure~$\varphi$.

Finally, over the remaining of Section~\ref{Sec 7-torus}, these ideas are applied to a~concrete example of certain
$T^3$-f\/ibrations over $T^4$, topologically equivalent to the 7-torus, which I will call \emph{${\rm G}_2$-torus
fibrations}~\cite{SaEarp2009}.
Deforming the metric (i.e.~the lattice) on $T^4$ induces a~change on the f\/ibration map and hence on the ${\rm
G}_2$-structure, and one can use Chern--Simons formalism to see how this af\/fects the moduli of ${\rm G}_2$-instantons:
\begin{theorem}
\label{thm: results}
Let $f:\mathbb{T} \rightarrow T^{4}$ be a~${\rm G}_2$-torus fibration, $\mathbf{E}\to\mathbb{T}$ be the pullback of an
indecomposable vector bundle $E\to T^4$ and~$\varphi$ denote the ${\rm G}_2$-structure of $\mathbb{T}$; then
\begin{enumerate}[$(i)$]\itemsep=0pt
\item every SD connection on~$E$ lifts to a~${\rm G}_2$-instanton on $\mathbf{E}$, i.e.,
\begin{gather*}
f^*\mathcal{M}^{4}_+\subset\mathcal{M}^7_{\varphi};
\end{gather*}
\item if, moreover, $c_2(E)\neq0$, then any perturbation $\varphi+\phi$ away from the class of fibred structures causes
the moduli space of ${\rm G}_2$-instantons to vanish, i.e.,
\begin{gather*}
\mathcal{M}^7_{\varphi+\phi}=\varnothing.
\end{gather*}
\end{enumerate}
\end{theorem}

The construction of ${\rm G}_2$-instantons is a~recent and active research area.
Indeed Theorem~\ref{thm: results} yields nontrivial, albeit nongeneric, examples of ${\rm G}_2$-instanton moduli,
whenever a~complex vector bundle $E\to T^4$ admits SD connections.
The interested reader will f\/ind other examples in works of Walpuski, Clarke and the
author~\cite{Clarke2013,SaEarp2009,SaEarp2011,SaEarp2013,Walpuski2011b}.
In the high-energy physics community, solutions to a~very similar problem in the context of ${\rm G}_2$-structures with
torsion have been found eg.
for cylinders over nearly-K\"ahler homogeneous spaces~\cite{Harland2010} and more generally for cones over nontrivial
manifolds admitting real Killing spinors~\cite{Ivanova2012}.

Finally, a~paper just published by Wang~\cite{Wang2012} makes signif\/icant progress towards a~Donaldson theory over
higher-dimensional foliations, which seems to encompass our ${\rm G}_2$-torus f\/ibration as a~special, codimension~4
tight foliation, whose leaf space is the smooth $4$-manifold~$X$.
It is inspiring to speculate whether an invariant of the corresponding foliated moduli space can be explicitly computed
for some suitable bundle $\mathbf{E}\to\mathbb{T}$, or indeed if that space coincides with our def\/inition of
$\mathcal{M}^7$.

\section[Gauge theory over ${\rm G}_2$-manifolds]{Gauge theory over $\boldsymbol{{\rm G}_2}$-manifolds}
\label{sec: gauge theory on G2-mfds}

I will concisely recall the essentials of gauge theory on ${\rm G}_2$-manifolds, while referring the interested reader
to a~more detailed exposition in~\cite{SaEarp2011}.

Let~$Y$ be an oriented smooth $7$-manifold; a~${\rm G}_2$-\emph{structure} is a~smooth $3$-form $\varphi \in
\Omega^{3}(Y) $ such that, at every point $p\in Y$, one has $\varphi_{p}=f_{p}^{\ast}(\varphi_{0})
$ for some frame $f_{p}:T_{p}Y\rightarrow \mathbb{R}^{7}$ and (adopting the conventions of~\cite{Salamon1989})
\begin{gather}
\label{eq:G2 3-form}
\varphi_0 =e^{567}+\omega_1\wedge e^{5} +\omega_2\wedge e^{6} +\omega_3\wedge e^{7}
\end{gather}
with
\begin{gather*}
\omega_1= e^{12} - e^{34},
\qquad
\omega_2= e^{13} - e^{42}
\qquad
\text{and}
\qquad
\omega_3= e^{14} - e^{23}.
\end{gather*}
Moreover,~$\varphi$ determines a~Riemannian metric $g(\varphi)$ induced by the pointwise inner-product
\begin{gather}
\label{eq:phi0 gives inner product}
 \langle u,v \rangle e^{1\dots 7}:=\frac{1}{6}  (u\lrcorner \varphi_{0} ) \wedge  (v\lrcorner
\varphi_{0} ) \wedge \varphi_{0},
\end{gather}
under which $\ast_\varphi\varphi$ is given pointwise~by
\begin{gather}
\label{eq:G2 4-form}
\ast \varphi_0 = e^{1234}-\omega_1\wedge e^{67} -\omega_2\wedge e^{75} -\omega_3\wedge e^{56}  .
\end{gather}
Such a~pair $(Y,\varphi ) $ is a~\emph{${\rm G}_2$-manifold} if $d\varphi =0$ and $d\ast_{\varphi}\varphi
=0$.

\subsection[The ${\rm G}_2$-instanton equation]{The $\boldsymbol{{\rm G}_2}$-instanton equation}

The ${\rm G}_2$-structure allows for a~$7$-dimensional analogue of conventional Yang--Mills
theory, yielding a~notion of (anti-)self-duality for $2$-forms.
Under the usual identif\/ication between $2$-forms and matrices, we have $\mathfrak{g}_{2}\subset
\mathfrak{so}(7) \simeq \Lambda^{2}$, so we denote $\Lambda_{14}^{2}:= \mathfrak{g}_{2}$ and
$\Lambda_{7}^{2}$ its orthogonal complement in $\Lambda^{2}$:
\begin{gather}
\label{eq:split}
\Lambda^{2} = \Lambda_{7}^{2}\oplus \Lambda_{14}^{2}.
\end{gather}
It is easy to check that $\Lambda_{7}^{2}= \langle e_1\lrcorner\varphi_0,\dots,e_7\lrcorner\varphi_0  \rangle$,
hence the orthogonal projection onto $\Lambda_{7}^{2}$ in~\eqref{eq:split} is given~by
\begin{alignat*}{3}
& L_{\ast \varphi_{0}}: \ \ && \Lambda^{2}  \rightarrow  \Lambda^{6},  &
\\
&&& \eta  \mapsto  \eta \wedge \ast \varphi_{0}&
\end{alignat*}
in the sense that~\cite[p.~541]{Bryant1985}
\begin{gather*}
L_{\ast \varphi_{0}} |_{\Lambda_{7}^{2}}: \ \ \Lambda_{7}^{2}
\; \tilde{\rightarrow}\;  \Lambda^{6}
\qquad
\text{and}
\qquad
L_{\ast \varphi_{0}}|_{\Lambda_{14}^{2}}=0.
\end{gather*}
Furthermore, since~\eqref{eq:split} splits $\Lambda^2$ into irreducible representations of ${\rm G}_2$, a~little
inspection on generators reveals that $(\Lambda^{2})_{^{7}_{14}}$ is respectively the $_{+1}^{-2}$-eigenspace of the $ {\rm G}_{2}$-equivariant linear map
\begin{alignat*}{3}
& T_{\varphi_0}: \ \ &&
\Lambda^{2}  \rightarrow  \Lambda^{2}, &&
\\
&&& \eta  \mapsto  T_{\varphi_0}\eta:=\ast  (\eta \wedge \varphi_{0} ). &
\end{alignat*}

Consider now a~vector bundle $E\rightarrow Y$ over a~compact ${\rm G}_2$-manifold $(Y,\varphi ) $; the
curvature $F:=F_{A} $ of some connection~$A$ decomposes according to the splitting~\eqref{eq:split}:
\begin{gather*}
F_A=F_{7}\oplus F_{14},
\qquad
F_{i}\in \Omega_{i}^{2} (\operatorname{End}E ),
\qquad
i=7,14.
\end{gather*}
The $L^2$-norm of $F_{A}$ is the \emph{Yang--Mills functional}:
\begin{gather}
\label{YM(A)}
{\rm YM}(A) := \Vert F_{A}\Vert^{2} =\Vert F_7\Vert^{2} +\Vert F_{14}\Vert^{2}.
\end{gather}

It is well-known that the values of ${\rm YM}(A)$ can be related to a~certain characteristic class of the
bundle~$E$, given (up to choice of orientation)~by
\begin{gather*}
\kappa  (E ):=-\int_{Y} \tr\big(F_{A}^{2}\big) \wedge \varphi.
\end{gather*}
Using the property $d\varphi =0$, a~standard argument of Chern--Weil theory~\cite{Milnor1974} shows that the de~Rham
class $[\tr(F_{A}^{2}) \wedge \varphi ]$ is independent of~$A$, thus the integral is indeed
a~topological invariant.
The eigenspace decomposition of $T_{\varphi}$ implies (up to a~sign)
\begin{gather*}
\kappa  (E ) = -2 \Vert F_{7} \Vert^{2}+ \Vert F_{14} \Vert^{2},
\end{gather*}
and combining with~\eqref{YM(A)} we get
\begin{gather*}
{\rm YM}(A)=-\frac{1}{2}\kappa (E)+\frac{3}{2}\Vert F_{14}\Vert^{2} =\kappa (E)+3 \Vert F_{7} \Vert^{2}.
\end{gather*}
Hence ${\rm YM}(A) $ attains its absolute minimum at a~connection whose curvature lies either in $\Lambda_{7}^{2}$
or in $\Lambda_{14}^{2}$.
Moreover, since ${\rm YM}\geq0$, the sign of $\kappa(E)$ obstructs the existence of one type or the other, so we f\/ix
$\kappa(E)\geq0$ and def\/ine \emph{${\rm G}_2$-instantons} as connections with $ F\in\Lambda_{14}^{2}$, i.e., such that
${\rm YM}(A)=\kappa(E)$.
These are precisely the solutions of the ${\rm G}_2$-\emph{instanton equation}:
\begin{subequations}
\begin{gather}
\label{eq:G2-intanton equation II}
F_A\wedge\ast\varphi=0
\end{gather}
or, equivalently,
\begin{gather}
\label{eq:G2-intanton equation}
F_A-\ast (F_A\wedge\varphi )=0.
\end{gather}
\end{subequations}
If instead $\kappa(E)\leq0$, we may still reverse orientation and consider $ F\in\Lambda_{14}^{2}$, but then the above
eigenvalues and energy bounds must be adjusted accordingly, which amounts to a~change of the $(-)$ sign in~\eqref{eq:G2-intanton equation}.

\subsection[Def\/inition of the Chern--Simons functional~$\vartheta$]{Def\/inition of the Chern--Simons functional~$\boldsymbol{\vartheta}$}\label{Subsect Chern-Simmons}

Gauge theory in higher dimensions can be formulated in terms of the geometric structure of manifolds with exceptional
holonomy~\cite{Donaldson1998}.
In particular, instantons can be characterised as critical points of a~Chern--Simons functional, hence zeroes of its
gradient 1-form~\cite{Donaldson2002}.
The explicit case of ${\rm G}_2$-manifolds, which we now describe, was f\/irst examined in the author's
thesis~\cite{SaEarp2009}.

Let $E\to Y$ be a~vector bundle; the space $\mathcal{A}$ is an af\/f\/ine space modelled on
$\Omega^{1}(\mathfrak{g}_{E}) $ so, f\/ixing a~reference connection $A_{0}\in \mathcal{A}$,
\begin{gather*}
\mathcal{A}=A_{0}+\Omega^{1} (\mathfrak{g}_{E} )
\end{gather*}
and, accordingly, vectors at $A\in\mathcal{A}$ are 1-forms $a,b,\ldots\in
T_A\mathcal{A}\simeq\Omega^{1}(\mathfrak{g}_{E}) $ and vector f\/ields are maps
$\alpha,\beta,\ldots:\mathcal{A}\to\Omega^{1}(\mathfrak{g}_{E})$.
In this notation we def\/ine the \emph{Chern--Simons functional}~by
\begin{gather*}
\vartheta  (A ):=\tfrac{1}{2}\int_{Y}\tr\left(d_{A_{0}}a\wedge
a+\frac{2}{3}a\wedge a\wedge a\right) \wedge \ast \varphi,
\end{gather*}
f\/ixing $ \vartheta (A_0) =0$.
This function is obtained by integration of the \emph{Chern--Simons $1$-form}
\begin{gather}
\label{ro^ phi}
\rho  (\beta  )_{A}=\rho_A(\beta_A):=\int_{Y} \tr (F_{A}\wedge
\beta_A ) \wedge \ast \varphi.
\end{gather}
We f\/ind $\vartheta $ explicitly by integrating $\rho $ over paths $A(t) =A_{0}+ta$, from $ A_{0}$ to any
$A=A_{0}+a$:
\begin{gather*}
\vartheta (A) -\vartheta (A_{0}) = \int_{0}^{1}\rho_{A(t)}  (\dot{A}(t) ) dt
 = \int_{0}^{1} \left(\int_{Y}\tr\big(\big(F_{A_{0}}+td_{A_{0}}a +t^{2}a\wedge a\big) \wedge a\big) \wedge \ast
\varphi \right)dt
\\
\phantom{\vartheta (A) -\vartheta (A_{0})}
 = \frac{1}{2}\int_{Y}\tr\left(d_{A_{0}}a\wedge a+\frac{2}{3} a\wedge a\wedge a\right) \wedge \ast \varphi +K,
\end{gather*}
where $K=K(A_0,a)$ is a~constant and vanishes if $A_0$ is an instanton.

The co-closedness condition $d\ast \varphi =0$ implies that the $1$-form~\eqref{ro^ phi} is closed, so the procedure
doesn't depend on the path $A(t)$.
Indeed, given tangent vectors $a,b\in\Omega^{1}(\mathfrak{g}_{E}) $ at~$A$, the leading term in the expansion
of $\rho$,
\begin{gather*}
\rho_{A+a}(b) -\rho_{A}(b) = \int_{Y}\tr (d_{A}a\wedge b ) \wedge \ast \varphi +
O\big( \vert b \vert^{2} \big),
\end{gather*}
is symmetric by Stokes' theorem:
\begin{gather*}
\int_{Y}\tr (d_{A}a\wedge b-a\wedge d_{A}b ) \wedge \ast \varphi =\int_{Y}d (\tr (b\wedge a )
\wedge \ast \varphi  ) =0.
\end{gather*}
We conclude that
\begin{gather*}
\rho_{A+a} (b)-\rho_{A}(b)=\rho_{A+b}(a) -\rho_{A}(a)+O\big(\vert b\vert^{2}\big)
\end{gather*}
and, comparing reciprocal Lie derivatives on parallel vector f\/ields $\alpha\equiv a$, $\beta\equiv b$ near a~point~$A$, we
have:
\begin{gather*}
d\rho (\alpha,\beta)_{A} =   (\mathcal{L}_{b}\rho )_{A}(a) - (\mathcal{L}_{a}\rho )_{A}(b)
 =  \lim\limits_{h\rightarrow 0}\frac{1}{h} \big\{ \rho_{A+hb} (a) -\rho_{A} (a) ) - (\rho_{A+ha} (b)
-\rho_{A}(b) ) \big\}
\\
\phantom{d\rho (\alpha,\beta)_{A}}
 =  \lim\limits_{h\rightarrow 0}\frac{1}{h^{2}}
\underset{O\big( \vert h \vert^{3}\big)} {\underbrace{\big\{ (\rho_{A+hb}  (ha ) -\rho_{A}
 (ha ) ) - (\rho_{A+ha}  (hb ) -\rho_{A}  (hb )  ) \big\}}}
 = 0.
\end{gather*}
Since $\mathcal{A}$ is contractible, by the Poincar\'{e} lemma $ \rho $ is the derivative of some function $\vartheta $.
Again by Stokes, $\rho $ vanishes along $\mathcal{G}$-orbits $\operatorname{im}(d_{A}) \simeq
T_{A} \{\mathcal{G}.A \} $.
Thus $\rho $ descends to the quotient $\mathcal{B}$ and so does $\vartheta $, locally.

\subsection[Periodicity of~$\vartheta$]{Periodicity of~$\boldsymbol{\vartheta}$}

Consider the gauge action of $g\in \mathcal{G}$ and some path $ \{A(t) \}_{t\in
[0,1]}\subset \mathcal{A}$ connecting an instanton~$A$ to~$g.A$.
The natural projection $p_{1}:Y\times [0,1] \rightarrow Y$ induces a~bundle
\begin{gather*}
\begin{array}{@{}ccc} \mathbf{E}_{g} & \overset{\widetilde{p_{1}}}{\longrightarrow} & E
\\
\downarrow & & \downarrow
\\
Y\times [0,1] & \overset{p_{1}}{\longrightarrow} & Y
\end{array}
\end{gather*}
and, using~$g$ to identify the f\/ibres $(\mathbf{E}_{g})_{0}
\overset{g}{\simeq}(\mathbf{E}_{g})_{1}$, one may think of $ \mathbf{E}_{g}$ as a~bundle over $Y\times
S^{1}$.
Moreover, in some local trivialisation, the path $A(t) =A_{i}(t) dx^{i}$ gives a~connection
$\mathbf{A}=\mathbf{A}_{0}dt+\mathbf{A}_{i}dx^{i}$ on $\mathbf{E}_{g}$:
\begin{gather*}
(\mathbf{A}_{0})_{(t,p)}=0,
\qquad
(\mathbf{A}_{i})_{(t,p)} =A_{i}(t)_{p}.
\end{gather*}
The corresponding curvature $2$-form is $F_{\mathbf{A}}=(F_{\mathbf{A}})_{i0}dx^{i}\wedge
dt+(F_{\mathbf{A}})_{jk}dx^{j}\wedge dx^{k}$, where
\begin{gather*}
(F_{\mathbf{A}})_{i0} =\dot{A}_{i}(t),
\qquad
(F_{\mathbf{A}})_{jk} =(F_{A})_{jk}.
\end{gather*}
The periods of~$\vartheta$ are then of the form
\begin{gather*}
\vartheta (g.A) -\vartheta (A)  =  \int_{0}^{1}\rho_{A(t)} (\dot{A}(t)) dt
 =  \int_{Y\times [0,1]}\tr(F_{A(t)} \wedge \dot{A}_{i}(t) dx^{i}) \wedge dt\wedge \ast\varphi
\\
\phantom{\vartheta (g.A) -\vartheta (A)}{}
 =  \int_{Y\times S^{1}}\tr F_{\mathbf{A}} \wedge F_{\mathbf{A}} \wedge\ast \varphi
 =  \frac{1}{8\pi^2} \big\langle c_{2} (\mathbf{E}_{g} ) \smallsmile  [\ast\varphi  ],Y\times S^{1}\big\rangle.
\end{gather*}
The K\"{u}nneth formula for $Y\times S^{1}$ gives
\begin{gather*}
H^{4}\big(Y\times S^{1},\mathbb{R}\big) =H^{4} (Y,\mathbb{R} ) \oplus H^{3} (Y,\mathbb{R} ) \otimes
\underset{\mathbb{Z}}{\underbrace{H^{1}\big(S^{1},\mathbb{R}\big)}}
\end{gather*}
and obviously $H^{4}(Y) \smallsmile [\ast \varphi ] =0$
so, denoting~by
$c_{2}^{\prime}(\mathbf{E}_{g}) $ the component lying in $H^{3}(Y) $ and~by
$S_{g}:=[\frac{1}{8\pi^2} c_{2}^{\prime}(\mathbf{E}_{g}) ]^{PD}$ its normalised Poincar\'e dual,
we are left with
\begin{gather*}
\vartheta (g.A) -\vartheta (A) = \langle  [\ast \varphi ],S_{g} \rangle.
\end{gather*}
Consequently, the periods of $\vartheta $ lie in the set
\begin{gather*}
\left\{ \int_{S_{g}}\ast \varphi
\,\bigg\vert\, S_{g}\in H_{4} (Y,\mathbb{R} ) \right\}.
\end{gather*}
That may seem odd at f\/irst, because $\ast\varphi$ is not, in general, an integral class and so the set of periods is
\emph{dense}.
However, as long as our interest remains in the study of the moduli space $\mathcal{M}=\operatorname{Crit}(\rho)$ of ${\rm
G}_2$-instantons, there is not much to worry, for the gradient $\rho =d\vartheta $ is unambiguously def\/ined on
$\mathcal{B}$.

\section[Instantons over ${\rm G}_2$-torus f\/ibrations]{Instantons over $\boldsymbol{{\rm G}_2}$-torus f\/ibrations}\label{Sec 7-torus}

Instances of ${\rm G}_2$-manifolds f\/ibred by associative submanifolds in the literature are relatively scarce, not least
because their deformation theory is zero-index elliptic~\cite{McLean1998} and therefore any new examples will be
somewhat exotic.
A~few trivial cases include the products $T^7=T^4\times T^3$ and $K3\times T^3$ and also $CY^3\times S^1$ given a~family
of curves in the Calabi--Yau~\cite[\S~10.8]{Joyce2000}.
The example I will propose is unique in the sense that the total space is not a~Riemannian product.

\subsection{Instantons over associative f\/ibrations}
\label{subsec: associative fibrations}

We consider pullback bundles over smooth associative f\/ibrations, and relate ${\rm G}_2$-instantons to their gauge theory
over the base; in particular we do not address the possibility of singular f\/ibres.

\begin{definition}
\label{def:assoc_fibration}
A~${\rm G}_2$-manifold $(Y^7,\varphi)$ is called an \emph{associative fibration} over a~compact oriented
Riemannian four-manifold $(X^4,\eta)$ if it is the total space of a~Riemannian submersion $f:Y\to X$ such that each
f\/ibre $ P_x:=f^{-1}(x)\subset Y$ is a~smooth associative submanifold.
\end{definition}

Since each f\/ibre $P_x$ is $3$-dimensional and orientable, its tangent bundle is dif\/ferentiably trivial and we may
choose global coordinates $t=(t^1,t^2,t^3)$ induced respectively by a~global coframe
$ \{e_5,e_6,e_7 \}:= \{dt^1,dt^2,dt^3 \}$.
Thus near each $y\in P_x$ we may complete the triplet into a~local orthogonal coframe $ \{e_1,\dots,e_7 \}$ of
$T^*Y$ such that $\varphi_y$ has the form~\eqref{eq:G2 3-form}, and the point~$y$ is unambiguously described~by
$(x,t(y))$.

\begin{lemma}
\label{lemma: SD lifts to G2-instanton}
Let $f:Y\to X$ define an associative fibration and $\mathbf{E}\to Y$ be the pullback from a~vector bundle $ E\to X$;
then a~connection~$A$ on~$E$ is self-dual if, and only if, $f^*A$ is a~${\rm G}_2$-instanton on $\mathbf{E}$.
\begin{proof}
Let $F:=(F_{f^*A})_y$ be the curvature 2-form at $y\in P_x$; then
\begin{gather*}
\ast_\varphi (F\wedge \varphi )  \overset{\rm loc}{=}  \ast_\varphi\big[F\wedge
\big(\varphi\vert_{P_x} +\omega_1\wedge e^5 +\omega_2\wedge e^6 +\omega_3\wedge e^7 \big)\big]
 \!=\!  \ast_\eta F\! + \ast_\varphi \big[O(F^{-})\wedge f^*\dvol_\eta\big],\!
\end{gather*}
where $O(F^{-}):=(F_{34}-F_{12} )e^5+(F_{42}-F_{13} )e^6+(F_{23}-F_{14} )e^7$ vanishes
precisely when~$A$ is self-dual, i.e., when $F=\ast_\eta F$ satisfies the ${\rm G}_2$-instanton equation~\eqref{eq:G2-intanton equation}.
\end{proof}
\end{lemma}

We are now in position to prove Theorem~\ref{thm: G2-instantons over assoc fibrations}.
Let us examine the general form of a~${\rm G}_2$-instanton on $\mathbf{E}$.
An arbitrary connection $\mathbf{A}$ on $\mathbf{E}$ is locally of the form
\begin{gather*}
\mathbf{A}(y) \overset{\rm loc}{=} A_t(x)+\sum\limits_{i=1}^3 \sigma_i(x,t)dt^i,
\end{gather*}
where $ \{A_t\}_{t\in t(P_x)}$ is a~family of connections on~$E$ and
$\sigma_i\in\Omega^0(Y,f^*\mathfrak{g}_E)$.
The curvature of $\mathbf{A}$ is
\begin{gather*}
F_\mathbf{A} = F_{A_t} +\sum\limits_{i=1}^3 \left(d_{A_t}\sigma_i-\frac{\partial A_t}{\partial t^i} \right) \wedge dt^i
+F_\sigma
\end{gather*}
with
\begin{gather*}
F_\sigma:= \sum\limits_{i,j=1}^3 \left(\frac{\partial \sigma_i}{\partial t^j}-\frac{\partial \sigma_j}{\partial
t^i}+\frac{1}{2} [\sigma_i,\sigma_j  ] \right) dt^i\wedge dt^j.
\end{gather*}
Replacing $F_\mathbf{A}$ into the ${\rm G}_2$-instanton equation~\eqref{eq:G2-intanton equation II} and using the
expression~\eqref{eq:G2 4-form} of $\ast\varphi$ in the natural frame $ \{e_1,\dots,e_7 \}$, we have
\begin{gather*}
\left(F_{A_t} + \sum\limits_{i=1}^3  (d_{A_t}\sigma_i- \frac{\partial A_t}{\partial t^i} ) \wedge e^{4+i} +
F_\sigma \right) \wedge \big(e^{1234}-\omega_1\wedge e^{67} -\omega_2\wedge e^{75} -\omega_3\wedge e^{56} \big) =0.
\end{gather*}
Using the following elementary properties
\begin{gather*}
  F_{A_t}\wedge e^{1234} =0,
\qquad
F_{A_t}\wedge \omega_1\wedge e^{67} = [(F_{A_t})_{34}-(F_{A_t})_{12}] \big({\ast}e^5\big),
\\
 F_{A_t}\wedge \omega_2\wedge e^{75} =[(F_{A_t})_{42}-(F_{A_t})_{13}]\big({\ast}e^6\big),
\qquad
F_{A_t}\wedge \omega_3\wedge e^{56}=[(F_{A_t})_{23}-(F_{A_t})_{14}]\big({\ast}e^7\big),
\\
  F_\sigma\wedge e^{4+i}\wedge e^{4+j}=0,
\qquad
F_\sigma\wedge e^{1234} = (F_\sigma)_{23}\big({\ast}e^5\big) +(F_\sigma)_{31}\big({\ast}e^6\big)+(F_\sigma)_{12}\big({\ast}e^7\big),
\end{gather*}
and the fact that each $d_{A_t}\sigma_i$ and $\frac{\partial A_t}{\partial t^i}$ are locally $1$-forms on the base,
hence their wedge product with $e^{1234}=\dvol_\eta$ vanishes, the equation simplif\/ies to
\begin{gather*}
\sum\limits_{i=1}^3 \left(d_{A_t}\sigma_i- \frac{\partial A_t}{\partial t^i} \right) \wedge \omega_i=0
\qquad
\text{and}
\qquad
  F_{A_t}^--Q (F_\sigma  )=0,
\end{gather*}
where~$Q$ is the linear map on $2$-forms def\/ined~by
\begin{gather*}
Q\big(dt^i\wedge dt^j\big) =Q\big(e^{4+i}\wedge e^{4+j}\big):=\sum\limits_{k=1}^3\epsilon^{ijk}\omega_k.
\end{gather*}
On the other hand, if $\mathbf{A}=A_t+\sum\sigma_i$ is a~${\rm G}_2$-instanton, then it minimises the Yang--Mills
functional~\eqref{YM(A)}.
This implies
\begin{gather*}
\sum\left\Vert d_{A_t}\sigma_i-\frac{\partial A_t}{\partial t^i}\right\Vert^2 +  \Vert F_\sigma \Vert^2 =0,
\end{gather*}
since otherwise the pullback component $A_t$ alone would violate the minimum energy:
\begin{gather*}
{\rm YM}(A_t)= \Vert F_{A_t} \Vert^2< \Vert F_{\mathbf{A}} \Vert^2={\rm YM}(\mathbf{A}).
\end{gather*}
In particular $F_\sigma\equiv0$ and so every $A_t$ must be SD.
Finally, if the family $A_t\equiv A_{t_0}$ is constant, then $d_{A_{t_0}}\sigma_i=0$ implies $\sigma\equiv0$, since~by
assumption $\mathbf{E}$ is indecomposable and therefore does not admit nontrivial parallel sections.
This concludes the proof of Theorem~\ref{thm: G2-instantons over assoc fibrations}.

\begin{remark}
If $\mathcal{M}^{4}_+$ is discrete, then by continuity the family $\{A_t\}$ is
contained in a~gauge orbit; if the family is constant, then $\mathbf{A}$ is a~pullback.
\end{remark}

\subsection[${\rm G}_2$-torus f\/ibrations]{$\boldsymbol{{\rm G}_2}$-torus f\/ibrations}

A~$7$-torus $T^{7}=\mathbb{R}^{7}/\Lambda $ naturally inherits the ${\rm G}_2$-structure $\varphi
$ from $\mathbb{R}^7$.
Recall from Section~\ref{Subsect Chern-Simmons} that a~connection~$A$ on some bundle over $T^7$ is a~${\rm
G}_2$-instanton if and only if it is a~zero of the Chern--Simons $1$-form~\eqref{ro^ phi}:
\begin{gather}
\label{Chern--Simons 1-form}
\rho_{A} (b) = \int_{T^7}\tr (F_{A}\wedge b ) \wedge\ast\varphi.
\end{gather}
One asks what is the behaviour of the moduli space of ${\rm G}_2$-instantons under perturbations
\mbox{$\varphi\rightarrow\varphi+\phi$} of the ${\rm G}_2$-structure.
More precisely, given suitable assumptions, one asks whether $(\varphi+\phi)$-instantons exist at all once we deform the
lattice.
As a~working example, we consider the following class of f\/lat $T^3$-f\/ibred 7-tori:
\begin{definition}
A~\emph{${\rm G}_2$-torus fibration structure} is a~triplet $(\eta,L,\alpha)$ in which:
\begin{itemize}\itemsep=0pt
\item $\eta$ is a~metric on $\mathbb{R}^4$;
\item $L$ is a~lattice on the subspace $\Lambda^2_{+}(\mathbb{R}^4,\eta)$ of~$\eta$-self-dual $2$-forms;
\item $\alpha:\mathbb{R}^4\rightarrow\Lambda^2_{+}(\mathbb{R}^4,\eta)$ is a~linear map.
\end{itemize}
\end{definition}

Given the above data, set $V \doteq \mathbb{R}^4 \oplus\Lambda^2_{+}$ and form the torus $\mathbb{T} = V/\tilde{L}$,
with the lattice
\begin{gather*}
\tilde{L} \doteq \big\{ (\mu,\nu+\alpha\mu ) \,|
\,
\mu\in \mathbb{Z}^4,
\,
\nu \in L \big\} \subset V.
\end{gather*}
Then $\mathbb{T}$ inherits from~$V$ the ${\rm G}_2$-structure~$\varphi$ which makes the generators of $\tilde{L}$
orthonormal with respect to the induced inner-product~\eqref{eq:phi0 gives inner product}.
It is straightforward to check that $\mathbb{T}$ is an associative f\/ibration as in Definition~\ref{def:assoc_fibration}:
denoting by $e^5$, $e^6$, $e^7$ the $(\nu+\alpha\mu)$-orthonormal basis of the f\/ibre~$\Lambda^2_{+}$, the f\/lat
${\rm G}_2$-structure~\eqref{eq:G2 3-form} simplif\/ies to
$\varphi|_{\Lambda^2_{+}}=e^{567}=\dvol_{\varphi}|_{\Lambda^2_{+}}$; moreover the lattice $\tilde{L}$ on every tangent
subspace normal to the f\/ibre is just the lattice~$\mu$ from the base, so the corresponding metrics are the same.
Although $\mathbb{T}$ f\/ibres over the 4-torus $ \mathbb{R}^4/\mu$, the induced metric $g(\varphi)$ is \emph{not}, in
general, a~Riemannian product.

Suppose the moduli space $\mathcal{M}^{4}_+$ of self-dual connections on $E\to T^4$ is nonempty; then we have trivial
solutions to the ${\rm G}_2$-instanton equation on the pullback $\mathbf{E}\to\mathbb{T}$ simply by lifting
$\mathcal{M}^{4}_+$ as in Lemma~\ref{lemma: SD lifts to G2-instanton}, which proves the f\/irst part of
Theorem~\ref{thm: results}:

\begin{corollary}
\label{cor:self-dual in T4 lifts to G2-instanton}
If~$A$ is a~self-dual connection on $ E\to T^4$, then its pullback $f^*A $ by the fibration map $f:\mathbb{T}
\rightarrow T^{4}$ is a~${\rm G}_2$-instanton on $\mathbf{E}$.
\end{corollary}

For future reference, I denote the set of such~$\varphi$-instantons obtained by lifts from $\mathcal{M}^{4}_+$ by
\index{moduli space}
\begin{gather}
\widetilde{\mathcal{M}^{4}_+}:= f^*\mathcal{M}^{4}_+ \subset\mathcal{B}^7.
\label{eq def Mtilde+}
\end{gather}

We know from $4$-dimensional gauge theory that SD connections on a~complex vector bundle $E\to T^4$ correspond to stable
holomorphic structures on~$E$, thus in such cases we have examples of ${\rm G}_2$-instantons on bundles over~$\mathbb{T}$.

\subsection[Deformations of $\mathbb{T}$]{Deformations of $\boldsymbol{\mathbb{T}}$}

Working on a~bundle $\mathbf{E}\rightarrow \mathbb{T}$ with compact structure group over a~f\/ixed ${\rm G}_2$-torus
f\/ibration, let us ponder in generality about the behaviour of instantons under a~deformation of the ${\rm G}_2$-structure:
\begin{gather*}
\varphi\rightarrow\varphi+\phi,
\qquad
*_\varphi\varphi\rightarrow *_\varphi\varphi+\xi_{\phi},
\qquad
\xi_{\phi}:= *_{\varphi+\phi}(\varphi+\phi)-*_\varphi\varphi \in\Omega^4  (\mathbb{T}  ).
\end{gather*}
An arbitrary deformation~$\phi$ does not in general preserve the f\/ibred structure of $\mathbb{T}$:
\begin{proposition}
\label{Prop: decomp deformations}
A~deformation $\xi_{\phi}\in\Lambda^4(\mathbb{T})$ of the coassociative $4$-form $*_\varphi\varphi$ has four orthogonal
components, with the following significance:
\begin{gather*}
\Lambda^4\big(\mathbb{R}^4\oplus \Lambda^2_+\big)= \underset{\rm (I)}{\underbrace{\Lambda^4\big(\mathbb{R}^4\big)}} \,\oplus\,
\underset{\rm (II )}{\underbrace{\Lambda^3\big(\mathbb{R}^4\big)\otimes \Lambda^1\big(\Lambda^2_+\big)}} \,\oplus\,
\underset{\rm (III)}{\underbrace{\Lambda^2\big(\mathbb{R}^4\big)\otimes \Lambda^2\big(\Lambda^2_+\big)}} \,\oplus\,
\underset{\rm (IV)}{\underbrace{\Lambda^1\big(\mathbb{R}^4\big)\otimes \Lambda^3\big(\Lambda^2_+\big)}},
\end{gather*}
\begin{enumerate}\itemsep=0pt
\item[{\rm (I)}] corresponds to a~rescaling of the metric~$\eta$ on $\mathbb{R}^4$;
\item[{\rm (II)}] redefines the map~$\alpha$;
\item[{\rm (III)}] splits as $\operatorname{Hom}(\Lambda^2_+,\Lambda^2_+)\oplus\operatorname{Hom}(\Lambda^2_-,\Lambda^2_+)$, where the first factor
modifies the lattice~$L$ and the second one affects the conformal class of~$\eta$;
\item[{\rm (IV)}] parametrises deformations transverse to the fibred structures.
\end{enumerate}

\begin{proof}
Let us examine the four cases.

(I) If $\xi_{\phi}\in\Lambda^4(\mathbb{R}^4)\simeq\mathbb{R}$, then it must be a~multiple of
$\ast\varphi|_{\mathbb{R}^4}=e^{1234}=\dvol_{\eta}$.

(II) Since $\Lambda^3(\mathbb{R}^4)\otimes \Lambda^1(\Lambda^2_+)\simeq \mathbb{R}^4\otimes
(\Lambda^2_+)^*\simeq\operatorname{Hom}(\mathbb{R}^4,\Lambda^2_+)$, such deformations are precisely linear maps
$\mathbb{R}^4\rightarrow\Lambda^2_{+}$.

(III) Clearly $\Lambda^2(\mathbb{R}^4)=\Lambda^2_+ \oplus \Lambda^2_-$ and $\Lambda^2(\Lambda^2_+)\simeq
(\Lambda^2_+)^*$, so the product decomposes as
\begin{gather*}
\big(\Lambda^2_+ \otimes \big(\Lambda^2_+\big)^* \big) \oplus \big(\Lambda^2_- \otimes \big(\Lambda^2_+\big)^*
\big)\simeq\operatorname{Hom}\big(\Lambda^2_+,\Lambda^2_+\big) \oplus \operatorname{Hom}\big(\Lambda^2_-,\Lambda^2_+\big).
\end{gather*}
Now, on one hand, acting with an endomorphism on $\Lambda^2_{+}$ is equivalent to redef\/ining the triplet
$\{e^5,e^6,e^7\}$, hence the lattice $L\subset\Lambda^2_{+}$.
On the other hand, since the orthogonal split $\Lambda^2= \Lambda^2_{-}\oplus\Lambda^2_{+}$ is conformally invariant,
a~map $\Lambda^2_{-}\to\Lambda^2_{+}$ redef\/ines the orthogonal complement of $\Lambda^2_-$ and hence the conformal
class.

(IV) Since $\Lambda^3(\Lambda^2_+)\simeq\mathbb{R}$, this component is just $\Lambda^1(\mathbb{R}^4)$, which is
irreducible in the sense that~$\mathbb{T}$ has no distinguished subspaces in $\mathbb{R}^4$.
Then either every $7$-torus is a~${\rm G}_2$-f\/ibration, which is obviously false, or these are precisely the
deformations away from said structures.
\end{proof}
\end{proposition}

We will now describe what happens to the zeroes of~\eqref{Chern--Simons 1-form} under the corresponding perturbation of
the Chern--Simons $1$-form:
\begin{gather*}
\rho \rightarrow \rho_{\phi}:= \rho +r_{\phi},
\qquad
(r_{\phi})_{A}(b) = \int_{\mathbb{T}} \tr (F_{A}\wedge b ) \wedge \xi_{\phi}.
\end{gather*}
Clearly a~$\varphi $-instanton~$A$ is also a~$(\varphi +\phi) $-instanton if and only if $(r_{\phi})_{A}\equiv 0$.
There is little reason, however, to expect such a~coincidence; as we will see, the topology of the bundle may constrain
the existence of instantons under certain~-- indeed most~-- deformations.

Denoting henceforth by $\mathcal{A}$ the space of connections over the $7$-manifold $\mathbb{T}$, let us brief\/ly digress
into the translation action of some vector $v \in \mathbb{T}$ on some $A\in \mathcal{A}$.
The f\/irst order variation is given by the bundle-valued 1-form
\begin{gather*}
(\beta_{v})_{A}:= v \lrcorner F_{A},
\end{gather*}
which we interpret as a~vector in $T_{A} \mathcal{A}$.
Notice f\/irst that in the direction $\beta_v$ the value of the Chern--Simons 1-form is independent of the base-point:

\begin{lemma}
\label{lemma: rho(bv) constant}
The function $\rho (\beta_v): \mathcal{A} \to \mathbb{R}$ is constant.
\begin{proof}
The computation is straightforward:
\begin{gather*}
\rho (\beta_v)_{A+ha}  =  \int_{\mathbb{T}}\tr F_{A+ha}\wedge v\lrcorner F_{A+ha}\wedge \ast\varphi
 =  -\frac{1}{2}\int_{\mathbb{T}}\tr F_{A+ha}\wedge F_{A+ha}\wedge (v\lrcorner\ast\varphi)
\\
\phantom{\rho  (\beta_v )_{A+ha}}
 =  -\frac{1}{2}\int_{\mathbb{T}} (\tr F_{A}\wedge F_{A}+d\chi )\wedge (v\lrcorner\ast\varphi)
 =  -\frac{1}{2}\int_{\mathbb{T}}\tr F_{A}\wedge F_{A}\wedge (v\lrcorner\ast\varphi)
 =  \rho (\beta_v)_{A},
\end{gather*}
where $d\chi$ is the exact differential given by Chern--Weil theory and we use Stokes' theorem and Cartan's identity $d
(v\lrcorner\ast\varphi)= \mathcal{L}_{v}(\ast\varphi)=0$, since~$\varphi$ is constant on the flat torus.
\end{proof}
\end{lemma}

Similarly, evaluating $r_{\phi}$ on $\beta_v$ gives
\begin{gather*}
r_{\phi}(\beta_v)_{A} = \int_{\mathbb{T}} \tr (F_{A}\wedge (\beta_v)_{A} ) \wedge\xi_{\phi}
 = -\tfrac{1}{2}\int_{\mathbb{T}}\tr (F_{A}\wedge F_{A} ) \wedge  (v\lrcorner \xi_{\phi}  )
 = \langle c_{2}(E),S_{\phi}(v)\rangle,
\end{gather*}
where $S_{\phi}(v) \doteq -\tfrac{1}{2}[v\lrcorner \xi_{\phi} ]^{PD}$, and this depends only on
the topology of~$E$, not on the point~$A$.

\begin{remark}
Hence we may interpret $\phi $ as def\/ining a~linear functional
\begin{alignat*}{3}
& N_{\phi}: \ \  && \mathbb{R}^7 \rightarrow  \mathbb{R}, & \\
&&& v\mapsto  \langle c_{2}(E),S_{\phi} (v)  \rangle, &
\end{alignat*}
such that $N_{\phi}\neq 0$ implies no~$\varphi$-instanton is still a~$(\varphi +\phi ) $-instanton.
This is, however, a~rather weak obstruction, since the map $\phi \mapsto N_{\phi}$ has kernel of dimension at least~$28$
and thus, in principle, leaves plenty of possibilities for instantons of perturbed ${\rm G}_2$-structures.
\end{remark}

Now consider specif\/ically a~translation vector on the base $v\in T^4$.
Notice that for deformations~$\phi$ of types (I), (II) or (III) the contraction of $ \xi_{\phi} $ with such~$v$ gives
$S_{\phi}(v) =0$, so~$\phi$ only ef\/fectively contributes to the function $\rho (\beta_v) $ when
$\xi_\phi\in\Lambda^1(\mathbb{R}^4)$, which means the perturbed torus is no longer a~f\/ibred structure
(Proposition~\ref{Prop: decomp deformations}).
Moreover, either the bundle~$E$ is f\/lat and~$\beta_v$ vanishes identically, or~$c_2(E)\neq0$ and the following holds:

\begin{lemma}
\label{lemma r_phi(bv) nonzero}
If $c_2(E)\neq0$ and~$\phi$ is of type {\rm (IV)}, then there exists $v\in T^4$ such that
$r_{\phi}(\beta_v)$ is a~non-zero constant.
\end{lemma}

\begin{proof}
Denoting $T^3$ the typical fibre of~$f$ (and setting $\operatorname{Vol}(T^3) =1$), we may assume
\begin{gather*}
\xi_\phi =-2\varepsilon \wedge \dvol_{T^3}
\end{gather*}
for some $0\neq\varepsilon \in \Lambda^{1}(T^{4}) $.
One can always choose $v\in T^{4}$ such that $\varepsilon (v)\neq0$, and consider
$(\beta_v)_{A}=v\lrcorner F_{A}$.
Then
\begin{gather*}
r_{\phi}(\beta_v)_{A} = -2\int_\mathbb{T}\tr (F_{A}\wedge v\lrcorner F_{A} ) \wedge \varepsilon\wedge d\operatorname{Vol}_{T^{3}}
 =  -2\int_{T^{4}}\tr (F_{A}\wedge v\lrcorner F_{A} ) \wedge \varepsilon
 =  \varepsilon(v) \cdot c_{2}(E),
\end{gather*}
which is nonzero by assumption.
\end{proof}

So far we know from Corollary~\ref{cor:self-dual in T4 lifts to G2-instanton} that the set $\mathcal{M}^{4}_+$ of
self-dual connections (modulo gauge) over $T^{4}$ lifts to instantons (cf.~\eqref{eq def Mtilde+}) of the original ${\rm
G}_2$-structure $\varphi $ (i.e.~to zeroes of~$\rho$).
However, for bundles with non-trivial $c_2$, this generic case degenerates precisely under deformations of type~(IV):

\begin{proposition}
\label{prop:no (phi+phi-instantons)}
Let $\mathbf{E}\rightarrow (\mathbb{T}, \varphi) $ be the pullback of a~stable ${\rm SU}(n)$-bundle~$E$ over $T^4$
with $c_{2}(E)\neq 0$; then~$E$ admits no $(\varphi +\phi ) $-instantons, for any perturbation $\phi $ away
from a~fibred structure $($i.e.~of type {\rm (IV)} in Proposition~{\rm \ref{Prop: decomp deformations})}.
\begin{proof}
Fix a~lifted~$\varphi$-instanton $A\in\widetilde{\mathcal{M}^{4}_+}$; for any $A+ha\in\mathcal{A}$,
Lemma~\ref{lemma: rho(bv) constant} gives $\rho_{A+ha}(\beta_v)\equiv\rho_A(\beta_v)=0$. Taking $v\in T^{4}$ as
in Lemma~\ref{lemma r_phi(bv) nonzero} we have
\begin{gather*}
\rho_{\phi}(\beta_v)_{A+ha}  =  r_{\phi}(\beta_v)_{A+ha} +\rho (\beta_v)_{A+ha}
 =  \underset{\neq0}{\underbrace{\varepsilon (v) \cdot c_{2}(E)}} + \underset{0}{\underbrace{\rho
(\beta_v)_{A}}},
\end{gather*}
hence $A+ha$ is not a~$(\varphi+\phi)$-instanton.
\end{proof}
\end{proposition}
Combining Corollary~\ref{cor:self-dual in T4 lifts to G2-instanton} and Proposition~\ref{prop:no
(phi+phi-instantons)} we obtain Theorem~\ref{thm: results}.

\subsection*{Acknowledgements}

The author thanks Simon Donaldson for suggesting the matter of this paper, Thomas Walpuski for sharing some unpublished
notes and Marcos Jardim for several useful discussions.
Special thanks also to the anonymous referees for numerous mathematical and reference contributions.

\pdfbookmark[1]{References}{ref}
\LastPageEnding

\end{document}